\documentclass[12pt,reqno]{amsart}
\usepackage{latexsym, amsmath, amscd, amssymb, amsthm,longtable} 
\usepackage[cp1251]{inputenc}
\usepackage[english]{babel}
\newtheorem{theorem}{Theorem}

\usepackage{mathtools}

\newcommand{\I}{\mathcal{I}}

\newcommand{\Map}{\operatorname{Map}}

\usepackage[all]{xy}
\begin{document}

\noindent

\title[ Joint Projective Invariants] { Joint Projective Invariants on First Jet Spaces of Point Configurations via Moving Frames} 

\author{Leonid Bedratyuk} 
\address{ Khmelnytsky National University, Ukraine}
\email{leonidbedratyuk@khmnu.edu.ua}

\begin{abstract}
We consider the action of the projective group $PGL(3,\mathbb{R})$ on the $n$-fold first-order jet space of point configurations on the plane. Using the method of moving frames, we construct an explicit complete generating set for the field of absolute first-order joint projective differential invariants $\mathcal{I}_{n,0}$ for any $n \ge 3$. This approach provides a unified construction for all $n$, immediately ensuring functional independence of the fundamental invariants and yielding formulas suitable for both symbolic and numerical implementation. 

Next, we study the field of relative first-order invariants $\mathcal{I}_n$ with Jacobian multiplier. It is shown that the invariantization of the Jacobian under the projective action yields a primitive element of the field extension $\mathcal{I}_n / \mathcal{I}_{n,0}$.

Finally, we introduce a multiplicative cochain complex $C^\bullet$ associated with the action of $PGL(3,\mathbb{R})$ on the jet space, and show that the invariantization operator induced by the moving frame generates an explicit contracting homotopy. This provides a constructive proof of the vanishing of higher cohomology and an interpretation of the “defect” of invariantization as an exact cocycle in $C^\bullet$.

\medskip


\end{abstract}

\maketitle

\section{Introduction}

We consider a smooth function $u\colon \mathbb{R}^2 \to \mathbb{R}$ interpreted as a scalar field on the plane. For each point $(x_i, y_i) \in \mathbb{R}^2$, we consider the first jet of $u$, that is, the tuple of variables $(x_i, y_i, p_i, q_i)$, where $p_i = \partial_x u(x_i, y_i)$, $q_i = \partial_y u(x_i, y_i)$. A configuration of $n$ such points defines an element of the $n$-fold first-order jet space:
\[
\mathcal{M} = \big(J^1(\mathbb{R}^2, \mathbb{R})\big)^{\!n} \;\cong\; \big(\mathbb{R}^2 \times \mathbb{R}^2\big)^{\!n},
\]
with coordinates of the form
\[
\boldsymbol{x} = \big(x_i, y_i, p_i, q_i\big)_{i=1}^n.
\]

The projective group $PGL(3, \mathbb{R})$ acts on the plane $\mathbb{R}^2$ via projective transformations. This action naturally prolongs to the jet space $J^1(\mathbb{R}^2, \mathbb{R})$ and hence to $\mathcal{M}$ by acting on each of the $n$ copies. The induced action of the projective group on $\mathcal{M}$ also defines its action on the field of rational functions $\mathbb{R}(\boldsymbol{x})$. In this field, a natural object of study is the field of joint differential projective relative invariants $\mathcal{I}^{(\mu)}_n$, consisting of elements that  transform according to the rule
\[
F(g \cdot \boldsymbol{x}) = \mu(g, \boldsymbol{x})^\omega \cdot F(\boldsymbol{x}), \quad \forall g \in G,\ \boldsymbol{x} \in \mathcal{M},
\]
where $\mu: G \times \mathcal{M} \to \mathbb{R}^\times$ is a non-vanishing multiplier and $\omega \in \mathbb{Q}$ is the \textit{weight} of the invariant. Invariants of zero weight generate the subfield $\mathcal{I}_{n,0}$ of \emph{absolute} invariants. In this paper, we focus on the case when the multiplier $\lambda(g, \boldsymbol{x})$ is the Jacobian of the projective transformation $g$, and the corresponding field of relative invariants is denoted simply by $\mathcal{I}_n$.

The problem we address here differs fundamentally from the classical problem of constructing joint projective differential invariants, which dates back to works from the late 19th century~\cite{Halphen1878}--\cite{Bouton1898}. In the classical approach, the objects of study are geometric submanifolds — parameterized curves or surfaces — and the corresponding invariants are functions that remain unchanged under the group action on the immersed submanifold in a certain ambient space. Thus, such invariants depend on the chosen parametrization or, at the very least, on the geometric structure of the submanifold.

In contrast, our work considers the action of the projective group $PGL(3, \mathbb{R})$ on scalar functions and their first derivatives at a finite set of points. That is, we deal not with invariants of submanifolds, but with invariants of the \emph{values of the function and its derivatives} at selected points — without invoking a global parametrization or geometric object.

Such invariants naturally arise in computer vision problems, where objects are represented not as smooth submanifolds but as raster images, i.e., discretized scalar fields. In this context, the function $u(x, y)$ is interpreted as the intensity function of a digital image, and its values at points $(x_i, y_i)$ represent the intensities of corresponding pixels. The partial derivatives $p_i = \partial_x u(x_i, y_i)$ and $q_i = \partial_y u(x_i, y_i)$ are approximated using local operators such as Sobel or Prewitt filters.

Moreover, knowledge of relative invariants of the projective group allows us to construct global numerical image descriptors that remain invariant under $PGL(3)$ transformations. Identifying an image with a smooth integrable function $u\colon \mathbb{R}^2 \to \mathbb{R}$, consider $F \in \mathcal{I}_n$ — a relative invariant with a \emph{Jacobian} multiplier and of weight $-1$. Then the $2n$-fold integral
\[
\int_{\mathbb{R}^{2n}} F(x_1,y_1,\dots,x_n,y_n) \,dx_1\,dy_1 \cdots dx_n\,dy_n,
\]
preserves its value under projective deformations of the image, i.e., it is a \textit{projectively invariant integral descriptor} of the image. Such image global hand-made features are crucial for pattern recognition tasks, where it is important to have features robust to geometric transformations of the scene. Complete systems of integral invariants have been thoroughly studied for the affine group and its subgroups~\cite{FSB}, but for the projective group, this problem remains open and requires new theoretical tools.

For the case of a single point, the projective differential invariants were recently fully described in~\cite{Olver2023} using the method of moving frames. However, the obtained invariants depend on higher-order derivatives, which makes them numerically unstable and unreliable for practical computations. Therefore, particular interest is given to invariants that depend only on first-order derivatives. For a configuration of two points, such invariants do not exist, whereas for configurations of three and four points, a single basic first-order differential invariant was constructed in~\cite{Wang}; see also~\cite{5}, \cite{1}, \cite{Open}.

In a recent article by the author~\cite{B-2025}, using a purely geometric approach, the problem of constructing joint first-order projective differential invariants was completely solved: for arbitrary $n$, explicit minimal generating sets of the fields of absolute and relative invariants $\mathcal{I}_{n,0}$ and $\mathcal{I}_n$ were found. Since in the aforementioned work the explicit formulas for the invariants were essentially derived heuristically, in this paper we return to the same problem, proposing an alternative — systematic, conceptually transparent, and effective — method of its solution: the method of \emph{moving frames}.

The method of moving frames, in its modern formulation developed in~\cite{Olver1999-1, Olver1999-2, Olver2007, Olver2011}; see also~\cite{Man}, is a universal tool for constructing invariants of Lie group actions on manifolds. Its fundamental advantage lies in the automatic guarantee of functional independence and minimality of the constructed invariants — unlike classical methods, which do not provide such guarantees.

It is worth noting that, as often happens, the pattern recognition community independently of mathematicians developed a simplified version of the moving frame method for subgroups of the affine group, known as the \textit{image normalization} method. This approach is widely used in image analysis as an effective tool for constructing affine image invariants. A survey of the image normalization method and its applications can be found in~\cite{Rot, FSB}.

In this paper, we show that applying the moving frame method to the description of the fields $\mathcal{I}_n$ and $\mathcal{I}_{n,0}$ significantly simplifies their construction. Moreover, it turns out that the invariantization of the Jacobian yields an explicit \emph{primitive element} of the algebraic extension $\mathcal{I}_n / \mathcal{I}_{n,0}$, while the invariantization of the cochain complex $C^\bullet$ leads to an explicit construction of a \emph{contracting homotopy operator}, from which the known fact of vanishing higher cohomology of the free projective group action follows constructively.

Thus, the proposed approach provides a complete system of absolute invariants in a compact, minimal, and structurally transparent form.

\medskip

The article is structured as follows. Section~2 provides a brief introduction to the method of moving frames, which is a constructive approach to building the field of invariants. 
Section~3 constructs a global moving frame for the projective group action on $\mathcal{M}$ and, using it, finds a minimal generating set for the field of absolute joint first-order projective differential invariants $\mathcal{I}_{n,0}$. 
Section~4 gives an explicit description of the field $\mathcal{I}_n$ of relative joint first-order differential invariants as a simple algebraic extension of the field $\mathcal{I}_{n,0}$. The primitive element of the extension is simply obtained from the moving frame invariantized Jacobian..  Section~4 carries out the construction of an invariantized contracting homotopy for the cochain complex $C^\bullet$, which allows a constructive proof of the vanishing of higher cohomology. This is done using a global right moving frame that provides explicit formulas for the homotopy operator.

\section{The Moving Frame Method}

In this section, let $\mathcal{M}$ be an arbitrary finite-dimensional manifold.  
The moving frame method is a constructive approach to the computation of the invariant field $\mathcal{O}(\mathcal{M})^G$ for a smooth action of an $r$-parameter Lie group $G$ on the manifold $\mathcal{M}$. Its power lies in reducing the construction of invariants to solving explicit, typically polynomial, equations for the group parameters, thereby shifting the focus from the representation $G$ to the geometry of orbits on $\mathcal{M}$. This section gives a brief exposition of the core constructions of the moving frame method needed for our subsequent analysis.

\subsection{Description of the Method}

We fix the class of functions $\mathcal{O}(\mathcal{M})$, the field of real rational functions on $\mathcal{M}$.  
A right action of $G$ on $\mathcal{M}$, $(g, \boldsymbol{x}) \mapsto g \cdot \boldsymbol{x}$, induces a contragredient action on functions:
\[
(g^{-1}F)(\boldsymbol{x}) = F(g \cdot \boldsymbol{x}), \qquad (g \in G,\ \boldsymbol{x} \in \mathcal{M}).
\]
The field of absolute invariants is then the fixed subfield:
\[
\mathcal{O}^G(\mathcal{M}) = \{\, F \in \mathcal{O}(\mathcal{M}) \ :\ g^{-1}F = F \quad \forall\, g \in G \,\},
\]
which is equivalent to the condition:
\[
F(g \cdot \boldsymbol{x}) = F(\boldsymbol{x}) \quad \forall\, g, \boldsymbol{x}.
\]

Thus, invariants are functions that take the same value at all points on a group orbit. The moving frame method constructs such functions via a suitable choice of an \emph{orbit cross-section} $\mathcal{K} \subset \mathcal{M}$ of co-dimension $r = \dim G$, which (locally) intersects each orbit transversely at exactly one point.

For each point $\boldsymbol{x} \in \mathcal{M}$, we seek an element $\rho(\boldsymbol{x}) \in G$ such that
\[
\rho(\boldsymbol{x}) \cdot \boldsymbol{x} \in \mathcal{K}.
\]
The resulting point $\rho(\boldsymbol{x}) \cdot \boldsymbol{x}$ is called the \emph{normalization} of $\boldsymbol{x}$ and depends only on the orbit to which $\boldsymbol{x}$ belongs. The smooth mapping that assigns to each point $\boldsymbol{x} \in \mathcal{M}$ an element $\rho(\boldsymbol{x}) \in G$ is called a (right) \textit{moving frame}. A moving frame exists in a neighborhood of a point $\boldsymbol{x}$ if and only if $G$ acts \textit{freely} and \textit{regularly} near $\boldsymbol{x}$.

Given local coordinates $\boldsymbol{x} = (x_1, \ldots, x_m)$ on $\mathcal{M}$, suppose the cross-section $\mathcal{K}$ is defined by the $r$ equations
\[
\mathcal{K} = \{\, x_1 = c_1,\ \dots,\ x_r = c_r \,\} \subset \mathcal{M},
\]
where $r = \dim G$ and $c_1, \dots, c_r$ are appropriately chosen constants. These equations fix $r$ coordinates of the normalized point. The associated right moving frame $\rho(\boldsymbol{x}) \in G$ is obtained by solving the \textit{normalization equations}
\[
(g \cdot \boldsymbol{x})_j = c_j, \qquad j = 1, \dots, r. \tag{$\ast$}
\]
This is a system of $r$ equations — typically polynomial — in $r$ unknowns, namely, the group parameters. Transversality combined with the Implicit Function Theorem implies the existence in some neighborhood $U \subset \mathcal{M}$ of a local solution to these algebraic equations. Then it defines $\rho(\boldsymbol{x})$ such that, for any point $\boldsymbol{x}$, the normalized point has coordinates of the form:
\[
\rho(\boldsymbol{x}) \cdot \boldsymbol{x} = \big(c_1, \dots, c_r,\ I_1(\boldsymbol{x}), \dots, I_{m - r}(\boldsymbol{x})\big),
\]
where the functions $I_1(\boldsymbol{x}), \dots, I_{m - r}(\boldsymbol{x})$, as will be shown below, are invariants of the group action.

The constructed moving frame satisfies the property of right $G$-equivariance — for all $g \in G$ and $\boldsymbol{x} \in U \cap g^{-1}U$, we have:
\[
\rho(g \cdot \boldsymbol{x}) = \rho(\boldsymbol{x})\, g^{-1}.
\]
Indeed, on one hand, by definition, $\rho(g \cdot \boldsymbol{x})$ maps $g \cdot \boldsymbol{x}$ to $\mathcal{K}$:
\[
\rho(g \cdot \boldsymbol{x}) \cdot (g \cdot \boldsymbol{x}) \in \mathcal{K}.
\]
On the other hand, since $\rho(\boldsymbol{x}) \cdot \boldsymbol{x} \in \mathcal{K}$, it follows that
\[
\big(\rho(\boldsymbol{x})\, g^{-1}\big) \cdot (g \cdot \boldsymbol{x}) = \rho(\boldsymbol{x}) \cdot \boldsymbol{x} \in \mathcal{K}.
\]
Since the solution to the normalization system is unique, and we have proposed two solutions, they must coincide:
\[
\rho(g \cdot \boldsymbol{x}) = \rho(\boldsymbol{x})\, g^{-1}.
\]

\medskip

It follows from the $G$-equivariance of $\rho$ that the coordinates of the normalized point are invariant. Let $\boldsymbol{y} = g \cdot \boldsymbol{x}$, then:
\[
\rho(\boldsymbol{y}) \cdot \boldsymbol{y}
= \rho(g \cdot \boldsymbol{x}) \cdot (g \cdot \boldsymbol{x})
= \big(\rho(\boldsymbol{x})\, g^{-1} \big) \cdot (g \cdot \boldsymbol{x}) 
= \rho(\boldsymbol{x}) \cdot \boldsymbol{x}.
\]
Hence, the normalized point remains unchanged along the orbit:
\[
I(\boldsymbol{y}) = I(\boldsymbol{x}).
\]

Therefore, the normalized coordinates not fixed by the cross-section,
\[
I_1(\boldsymbol{x}), \dots, I_{m - r}(\boldsymbol{x}),
\]
are \emph{invariants} of the $G$-action; that is, for all $g \in G$ and $k = 1,\dots,m - r$, we have:
\[
I_k(g \cdot \boldsymbol{x}) = I_k(\boldsymbol{x}).
\]

Similarly, for a left moving frame, the normalized point is $\rho(\boldsymbol{x})^{-1} \cdot \boldsymbol{x}$ and the equivariance condition takes the form:
\[
\rho(g \cdot \boldsymbol{x}) = g \rho(\boldsymbol{x}).
\]
If $g \cdot \boldsymbol{x}$ is a right-equivariant moving frame, then $(g \cdot \boldsymbol{x})^{-1}$ is left-equivariant, and vice versa.

The choice of a cross-section $\mathcal{K} \subset \mathcal{M}$ uniquely determines the moving frame $\rho$, which in turn initiates the process of \emph{invariantization} — a canonical transition from smooth functions to invariants.

The invariantization of a function $F : \mathcal{M} \to \mathbb{R}$ is defined by the formula
\[
\iota(F)(\boldsymbol{x}) := F\big(\rho(\boldsymbol{x}) \cdot \boldsymbol{x}\big).
\]
Thus, $\iota(F)$ is a $G$-invariant function that agrees with $F$ on the cross-section $\mathcal{K}$:
\[
\iota(F)\big|_{\mathcal{K}} = F\big|_{\mathcal{K}}.
\]
In particular, invariantizing the coordinate functions $x_j$ yields:
\[
\iota(x_j) = c_j,\quad j = 1,\dots,r,
\qquad
I_k := \iota(x_{r+k}),\quad k = 1,\dots,m-r,
\]
where $I_1,\dots,I_{m-r}$ are the \emph{fundamental invariants}, and $\iota(x_1),\dots,\iota(x_r)$ are the so-called \emph{phantom invariants}, equal to the fixed coordinates of the cross-section.

The \emph{Replacement Theorem} guarantees (locally) that any smooth invariant is a function of the fundamental invariants. Depending on the choice of cross-section $\mathcal{K}$, the set of fundamental invariants may vary, but any two such sets define the same field of $G$-invariants.

The number of fundamental invariants, i.e., the transcendence degree of the field $\mathcal O(\mathcal M)^G$ for a free group action, according to Rosenlicht’s theorem (cf.~\cite{Ros},~\cite{Sp}), is:
\[
\operatorname{trdeg} \mathcal O(\mathcal M)^G = \dim \mathcal M  - \dim G.
\]

If the action of $G$ on $\mathcal M$ is not free or not regular, a moving frame still exists for the prolonged action on jet spaces  or for the \emph{joint action} on Cartesian powers (joint invariants). In both approaches, a suitable choice of $n$ or $k$ renders the action (locally) free and regular, after which the standard normalization procedure applies.

A key advantage of the moving frame method over other methods for constructing invariants is that the resulting fundamental invariants are always functionally independent, and any other invariant is an analytic function of them. This is an exceptionally strong property. In classical invariant theory, the main difficulty lies not so much in constructing individual invariants — for which many effective algorithms exist — as in isolating a minimal rational or polynomially independent generating set. The moving frame method is unique in that it immediately, as a consequence of normalization, produces a complete system of functionally independent invariants, usually rational. To our knowledge, no other method of constructing invariants possesses this property.


\section{Joint First-Order Absolute Projective Differential Invariants}

In this section, we construct an explicit \emph{complete} generating system for the field of absolute joint projective differential invariants of the first order, denoted $\mathcal{I}_{n,0}$. Unlike the approach in~\cite{B-2025}, where heuristically chosen systems were proposed separately for small values $n = 2,3,4,5,6$ and for $n \geq 7$, and their functional independence was proved independently, we employ the method of moving frames: the choice of a cross-section immediately yields functionally independent fundamental invariants for any $n \ge 3$.

\subsection{Construction of the Moving Frame}

We consider the $n$-fold first-order joint jet space
\[
\mathcal{M}
\cong \big(\mathbb{R}^2\times\mathbb{R}^2\big)^{\!n},
\]
with coordinates
\[
\boldsymbol{x}=\big(x_i,y_i,p_i,q_i\big)_{i=1}^n,
\qquad p_i=\partial_x u|_{(x_i,y_i)},\;\; q_i=\partial_y u|_{(x_i,y_i)}.
\]
Let $\mathcal{I}_{n,0} = \mathcal{O}(\mathcal{M})^{PGL(3, \mathbb{R})}$. Denote
\[
T=\begin{pmatrix}
a_1 & a_2 & a_3\\
b_1 & b_2 & b_3\\
c_1 & c_2 & c_3
\end{pmatrix}\in GL(3,\mathbb{R}),
\quad g=[T]\in PGL(3,\mathbb{R}).
\]
Recall that for each pair of variables $(x_i,y_i), i = 1,2,\dots, n$, the group action has the form:
\begin{gather*}
g(x_i)=\widetilde{x}_i= \frac {x_i a_1 + y_i a_2 + a_3}{x_i c_1 + y_i c_2 + 1},\\
g(y_i)=\widetilde{y}_i= \frac {x_i b_1 + y_i b_2 + b_3}{x_i c_1 + y_i c_2 + 1}.
\end{gather*}

The first prolongations are given by:
\begin{align*}
\tilde p_i &= \frac{c_1 x_i + c_2 y_i + 1}{\det T} \left(
  -\det\begin{vmatrix} b_1 & b_2 \\ c_1 & c_2 \end{vmatrix} (p_i x_i + q_i y_i)
  + \det\begin{vmatrix} b_2 & b_3 \\ c_2 & c_3 \end{vmatrix} p_i
  + \det\begin{vmatrix} b_1 & b_3 \\ c_1 & c_3 \end{vmatrix} q_i
\right), \\[4pt]
\tilde q_i &= \frac{c_1 x_i + c_2 y_i + 1}{\det T} \left(
  \det\begin{vmatrix} a_1 & a_2 \\ c_1 & c_2 \end{vmatrix} (p_i x_i + q_i y_i)
  - \det\begin{vmatrix} a_2 & a_3 \\ c_2 & c_3 \end{vmatrix} p_i
  + \det\begin{vmatrix} a_1 & a_3 \\ c_1 & c_3 \end{vmatrix} q_i
\right),
\end{align*}
see e.g. \cite{B-2025}.

The group $PGL(3,\mathbb{R})$ has dimension 8. On $\mathcal{M}$, we choose a cross-section of co-dimension 8 by normalizing three points and one gradient vector:
\[
\mathcal{K} = \{ x_1 = 1,\ x_2 = 0,\ x_3 = 0,\ y_1 = 0,\ y_2 = 0,\ y_3 = 1,\ p_1 = 1,\ q_1 = 0 \}.
\]
Note that this cross-section is not symmetric with respect to the variables, so the resulting expressions for the group parameters and the invariants will also be asymmetric.

Geometrically, we fix a projective frame from three non-collinear points
\[
(X_1,X_2,X_3) = ((1,0), (0,0), (0,1))
\]
and set the orientation and scale of the axis at point $X_1$ via gradient normalization.

The requirement $g \cdot \boldsymbol{x} \in \mathcal{K}$ is equivalent to the following system of equations for the group parameters:

$$
\begin{cases}
\displaystyle {\frac {a_{{1}}x_{{1}}+a_{{2}}y_{{1}}+a_{{3}}}{c_{{1}}x_{{1}}+c_{{2}}y
_{{1}}+1}}
=1, \qquad  \displaystyle {\frac {b_{{1}}x_{{3}}+b_{{2}}y_{{3}}+b_{{3}}}{c_{{1}}x_{{3}}+c_{{2}}y
_{{3}}+1}}=1,\\

a_{{1}}x_{{3}}+a_{{2}}y_{{3}}+a_{{3}}=0, \qquad
b_{{1}}x_{{1}}+b_{{2}}y_{{1}}+b_{{3}}=0,\\
b_{{1}}x_{{2}}+b_{{2}}y_{{2}}+b_{{3}}=0, \qquad
a_{{1}}x_{{2}}+a_{{2}}y_{{2}}+a_{{3}}=0,\\
\displaystyle \frac{c_{{1}}x_{{1}}+c_{{2}}y_{{1}}+1}{\delta}(\left( b_{{1}}x_{{1}}c_{{2}}-b_{{2}}x_{{1}}c_{{1}}+b_{{3}}c_{{2}}-b_{{2}} \right) p_{{1}}+ \left( y_{{1}}b_{{1}}c_{{2}}-b_{{2}}y_{{1}}c_{{1
}}-b_{{3}}c_{{1}}+b_{{1}} \right) q_{{1}})=1,\\
\displaystyle \left( a_{{1}}x_{{1}}c_{{2}}-a_{{2}}x_{{1}}c_{{1}}+a_{{3}}c_{{2}}-a_{{2}} \right) p_{{1}}+ \left( y_{{1}}a_{{1}}c_{{2}}-a_{{2}}y_{{1}}c_{{1
}}-a_{{3}}c_{{1}}+a_{{1}} \right) q_{{1}}=0.
\end{cases}
$$

We have omitted the denominators for expressions that are set to zero.

Solving this system, we obtain the values of the group parameters:

\begin{gather*}
a_{{1}}={\frac {y_{{2}}-y_{{3}}}{\delta\,q_1 y_{{1}}+\delta\,p_1 x_{{1}}+x_{{2}}y_{{3}}-x_{{3}}y_{{2}}}}
\\
a_{{2}}={\frac {-x_{{2}}+x_{{3}}}{\delta\,q_1y_{{1}}+\delta\,u_{
{1,0}}x_{{1}}+x_{{2}}y_{{3}}-x_{{3}}y_{{2}}}}
\\
a_{{3}}={\frac {x_{{2}}y_{{3}}-x_{{3}}y_{{2}}}{\delta\,q_1y_{{1}
}+\delta\,p_1x_{{1}}+x_{{2}}y_{{3}}-x_{{3}}y_{{2}}}}
\\
b_{{1}}={\frac { \left( y_{{1}}-y_{{2}} \right)  \left( y_{{1}}-y_{{3}
} \right) q_1+ \left( y_{{1}}-y_{{2}} \right)  \left( x_{{1}}-x_
{{3}} \right) p_1}{\delta\,q_1y_{{1}}+\delta\,p_1x_{
{1}}+x_{{2}}y_{{3}}-x_{{3}}y_{{2}}}}
\\
b_{{2}}={\frac {- \left( x_{{1}}-x_{{2}} \right)  \left( y_{{1}}-y_{{3
}} \right) q_1- \left( x_{{1}}-x_{{2}} \right)  \left( x_{{1}}-x
_{{3}} \right) p_1}{\delta\,q_1y_{{1}}+\delta\,p_1x_
{{1}}+x_{{2}}y_{{3}}-x_{{3}}y_{{2}}}}
\\
b_{{3}}={\frac { \left( x_{{1}}y_{{2}}-x_{{2}}y_{{1}} \right)  \left( 
y_{{1}}-y_{{3}} \right) q_1+ \left( x_{{1}}y_{{2}}-x_{{2}}y_{{1}
} \right)  \left( x_{{1}}-x_{{3}} \right) p_1}{\delta\,q_1
y_{{1}}+\delta\,p_1x_{{1}}+x_{{2}}y_{{3}}-x_{{3}}y_{{2}}}}
\\
c_1=\frac{y_{{2}}-y_{{3}}-p_1{\delta}}{{\delta}(q_1y_{{1}}+p_1x_{{1}})+x_{{2}}y_{{3
}}-x_{{3}}y_{{2}}
}\\
c_2=\frac{-q_1{\delta}-x_{{2}}+x_{{3}}}{{\delta}(q_1y_{{1}}+p_1x_{{1}})+x_{{2}}y_{{3
}}-x_{{3}}y_{{2}}
}
\end{gather*}
where  $\delta=\delta_{1,2,3}$  and   $ \delta_{i,j,k}=\begin{vmatrix} x_i & x_j & x_k \\ y_i & y_j & y_k \\ 1 & 1 & 1 \end{vmatrix}.$

The obtained solution defines a global moving frame
\[
\rho:\ \mathcal{M} \to PGL(3,\mathbb{R}),
\qquad 
\rho(g\cdot\boldsymbol{x}) = \rho(\boldsymbol{x})\,g^{-1}.
\]

The normalized point is the result of applying $\rho$ to the original coordinates:
\[
\rho(\boldsymbol{x}) \cdot \boldsymbol{x}
=
\Big(
1,0,\,1,0, 0,0,\iota( p_2),\iota(q_2), 0,1,\,\iota( p_3),\iota(q_3), 
\underbrace{\iota(x_i),\iota( y_i),\iota( p_i),\iota(q_i)}_{i=4,\dots,n}
\Big),
\]
where $\iota$ denotes invariantization defined earlier.

Since the cross-section uses the first three points, it is convenient for further descriptions to treat separately the cases $n=3$ and $n \ge 4$.

\subsection{The Case \(n=3\)}
Since $\dim \mathcal{M}=4\cdot 3=12$, we have
\[
\operatorname{trdeg}\,\mathcal{I}_{3,0}=12-8=4.
\]
Eight coordinates are used for normalization, so after substituting the frame parameters, the remaining four coordinates are invariantized into four fundamental invariants. It is convenient to introduce the notation
\[
\Phi^{(k)}_{ij}:=(x_i-x_j)p_k+(y_i-y_j)q_k,\qquad i\ne j.
\]

\begin{theorem}
\label{thm:n3}
The field of absolute first-order joint projective invariants is the rational function field
\[
\mathcal{I}_{3,0}
\;=\;
\mathbb{R}\bigl(\zeta_{12},\zeta_{23},\zeta_{13},\tau \bigr),
\]
where
$$
\begin{aligned}
&\zeta_{12}=\Phi^{(1)}_{12} \Phi^{(2)}_{12}, 
\zeta_{23}=\Phi^{(2)}_{23} \Phi^{(3)}_{23},\\
&\zeta_{13}=\Phi^{(1)}_{13} \Phi^{(3)}_{13},
\tau=\Phi^{(1)}_{13}\Phi^{(3)}_{23} \Phi^{(2)}_{12},
\end{aligned}
$$
are algebraically independent absolute first-order joint projective invariants. 
\end{theorem}

\begin{proof}
Invariantization of the four “free” coordinates $p_2,q_2,p_3,q_3$ (i.e., substitution of the parameters $\rho$ into the prolongation formulas) gives:

\begin{gather*}
\iota( p_2)=   \left(  \left( x_{{1}}-x_{{2}} \right) p_{{
1}}+ \left( y_{{1}}-y_{{2}} \right) q_{{1}} \right) \left(  \left( x_{{1}}-x_{{2}} \right) p_{{2}}+ \left( y_{{1}}-y_{{2}} \right) q_{{2}} \right)=\Phi^{(1)}_{12} \Phi^{(2)}_{12} \,\\
\iota(q_2)=-{\frac { \left(  \left( x_{{1}}-x_{{2}} \right) p_{{1}}+ \left( y_{{1
}}-y_{{2}} \right) q_{{1}} \right)  \left(  \left( x_{{2}}-x_{{3}}
 \right) p_{{2}}+ \left( y_{{2}}-y_{{3}} \right) q_{{2}} \right) }{
 \left( x_{{1}}-x_{{3}} \right) p_{{1}}+ \left( y_{{1}}-y_{{3}}
 \right) q_{{1}}}}=\frac{\Phi^{(1)}_{12} \Phi^{(2)}_{23} }{\Phi^{(1)}_{13} }
,\\
\iota( p_3)=\frac{  \left(  \left( x_{{1}}-x_{{3}} \right) p_{{3}}+ \left( y_{{1}}-y_{{3}
} \right) q_{{3}} \right)  \left(  \left( x_{{1}}-x_{{2}} \right) p_{{
1}}+ \left( y_{{1}}-y_{{2}} \right) q_{{1}} \right) - \left( x_{{2}}-x
_{{3}} \right) p_{{3}}- \left( y_{{2}}-y_{{3}} \right) q_{{3}}
}{\left( x_{{1}}-x_{{2}} \right) p_{{1}}+ \left( y_{{1}}-y_{{2}} \right) q_{{1}}
} \times \\ \times  (\left( x_{{1}}-x_{{3}} \right) p_{{1}}+ \left( y_{{1}}-y_{{3}} \right) q_{{1}})=\frac{\Phi^{(1)}_{13} \left( \Phi^{(3)}_{13} \Phi^{(1)}_{12} -\Phi^{(3)}_{23}  \right) }{\Phi^{(1)}_{12} },\\
\iota( q_3)={\frac { \left(  \left( x_{{2}}-x_{{3}} \right) p_{{3}}+ \left( y_{{2}}-y_{{3}} \right) q_{{3}} \right)  \left(  \left( x_{{1}}-x_{{3}}
 \right) p_{{1}}+ \left( y_{{1}}-y_{{3}} \right) q_{{1}} \right) }{
 \left( x_{{1}}-x_{{2}} \right) p_{{1}}+ \left( y_{{1}}-y_{{2}}
 \right) q_{{1}}}}=\frac{\Phi^{(1)}_{13}\Phi^{(3)}_{23}}{\Phi^{(1)}_{12}}.
\end{gather*}

These invariants are algebraically independent and generate the field $\mathcal{I}_{3,0}$. Let us choose a simpler basis of the field. One easily sees the following relations:
\begin{gather*}
\iota(\tilde q_2) \iota(\tilde q_3)=  \Phi^{(2)}_{23} \Phi^{(3)}_{23},\\
\iota(\tilde p_3)=\Phi^{(1)}_{13} \Phi^{(3)}_{13} -\iota(\tilde q_3),\\
\iota(\tilde p_2) \iota(\tilde q_3)= \Phi^{(1)}_{13}\Phi^{(3)}_{23} \Phi^{(2)}_{12}
\end{gather*}
These relations allow us to introduce new, simpler invariants:
\begin{gather*}
\zeta_{1}^{(2)}=\Phi^{(1)}_{12} \Phi^{(2)}_{12}, \quad
\zeta_{1}^{(3)}=\Phi^{(1)}_{13} \Phi^{(3)}_{13}, \quad
\zeta_{2}^{(3)}=\Phi^{(2)}_{23} \Phi^{(3)}_{23}, \quad
\tau=\Phi^{(1)}_{13}\Phi^{(3)}_{23} \Phi^{(2)}_{12},
\end{gather*}
which are related to the invariants $\iota(\tilde p_2),\iota(\tilde q_2),\iota(\tilde p_3),\iota(\tilde q_3)$ via a non-degenerate transformation. Therefore, the functional independence of $\iota(p_2),\iota(q_2),\iota(p_3),\iota(q_3)$ is equivalent to the functional independence of $\zeta_{1}^{2},\zeta_{1}^{3},\zeta_{2}^{3},\tau$, which proves the structure of the field $\mathcal{I}_{3,0}$.
\end{proof}

Although the invariance of the given expressions is ensured by the moving frame method itself, it is useful to see the mechanism of its appearance — a simple mutual cancellation of the multipliers of the projective action. Direct computation shows that
\begin{gather*}
g(\Phi^{(i)}_{i,j})= \frac{c_1 x_i+c_2 x_i+1}{c_1 x_j+c_2 x_j+1}  \, \Phi^{(i)}_{i,j},\\
g(\Phi^{(j)}_{i,j})= \frac{c_1 x_j+c_2 x_j+1}{c_1 x_i+c_2 x_i+1} \, \Phi^{(j)}_{i,j}.
\end{gather*}

Hence, the invariance of the required combinations follows immediately due to mutual cancellation of the multipliers, for example:
\begin{gather*}
g(\tau)=g(\Phi^{(1)}_{13}) g(\Phi^{(3)}_{23}) g(\Phi^{(2)}_{12})= \frac{c_1 x_1+c_2 x_1+1}{c_1 x_3+c_2 x_3+1} \cdot \frac{c_1 x_3+c_2 x_3+1}{c_1 x_2+c_2 x_2+1} \times \\ \times \frac{c_1 x_2+c_2 x_2+1}{c_1 x_1+c_2 x_1+1} \cdot \Phi^{(1)}_{13}\Phi^{(3)}_{23} \Phi^{(2)}_{12}=\Phi^{(1)}_{13}\Phi^{(3)}_{23} \Phi^{(2)}_{12}.
\end{gather*}

In the paper \cite{B-2025}, using a geometric approach, an almost identical polynomial basis of the field $\mathcal I_{3,0}$ is found in the form of determinants:
\[
\begin{aligned}
\zeta_{1}^{(2)} &=
\left| \begin {array}{ccc} p_{{1}}&q_{{1}}&-p_{{1}}x_{{1}}-q_{{1}}y_{{1}}\\ \noalign{\medskip}p_{{2}}&q_{{2}}&-p_{{2}}x_{{2}}-q_{{2}}y_{{2}
}\\ \noalign{\medskip}y_{{1}}-y_{{2}}&x_{{2}}-x_{{1}}&x_{{1}}y_{{2}}-x
_{{2}}y_{{1}}\end {array} \right|,\\
\zeta_{2}^{(3)} &=
 \left| \begin {array}{ccc} p_{{2}}&q_{{2}}&-p_{{2}}x_{{2}}-q_{{2}}y_{{2}}\\ \noalign{\medskip}p_{{3}}&q_{{3}}&-p_{{3}}x_{{3}}-q_{{3}}y_{{3}
}\\ \noalign{\medskip}y_{{2}}-y_{{3}}&-x_{{2}}+x_{{3}}&x_{{2}}y_{{3}}-
x_{{3}}y_{{2}}\end {array} \right| 
,\\
\zeta_{1}^{(3)} &=
 \left| \begin {array}{ccc} p_{{1}}&q_{{1}}&-p_{{1}}x_{{1}}-q_{{1}}y_{{1}}\\ \noalign{\medskip}p_{{3}}&q_{{3}}&-p_{{3}}x_{{3}}-q_{{3}}y_{{3}
}\\ \noalign{\medskip}y_{{1}}-y_{{3}}&x_{{3}}-x_{{1}}&x_{{1}}y_{{3}}-x
_{{3}}y_{{1}}\end {array} \right|,\\[6pt]
\end{aligned}
\]
$$
\tau'= 
\begin{vmatrix}
 p_1 & q_1 & p_1x_1+q_1y_1\\
 p_2 & q_2 & p_2x_2+q_2y_2\\
 p_3 & q_3 & p_3x_3+q_3y_3
\end{vmatrix}\cdot \begin{vmatrix}
 x_1 & y_1 & 1\\
 x_2 & y_2 & 1\\
 x_3 & y_3 & 1
\end{vmatrix}.
$$

It can be shown that $\tau'$ is expressed via the basis as
\[
\tau'=\tau+\frac{\zeta_{1}^{(2)} \zeta_{1}^{(3)} \zeta_{2}^{(3)} }{\tau}.
\]

Thus, both bases — the geometric one from \cite{B-2025}  and the one obtained by the moving frame method — generate the same field \(\mathcal I_{3,0}\).


\subsection{The Case \boldmath{$n=4$}}

For four points in general position, we have:
\[
\operatorname{trdeg} \mathcal{I}_{4,0} = 8.
\]
The four invariants constructed in the case \(n=3\),
\[
\zeta_{1}^{(2)},\quad \zeta_{2}^{(3)},\quad \zeta_{1}^{(3)},\quad \tau,
\]
remain invariants also for \(n=4\). The additional four invariants are obtained by invariantization of the coordinates of the fourth block:
\[
\iota(x_4),\qquad \iota(y_4),\qquad \iota(p_4),\qquad \iota(q_4),
\]
which can be expressed via determinants and the functions \(\Phi\). We have:
\begin{gather*}
\iota(x_4)=\frac{\delta_{234}}{\delta_{234}+\delta_{123} \Phi^{(1)}_{1,4}},\\
\iota(y_4)=\frac{\delta_{124} \Phi^{(1)}_{1,3}}{\delta_{234}+\delta_{123} \Phi^{(1)}_{1,4}},\\
\iota(p_4)=\frac{(\delta_{234}+\delta_{123} \Phi^{(1)}_{1,2}) (\Phi^{(4)}_{1,4} \Phi^{(1)}_{1,2}-\Phi^{(4)}_{2,4})}{\delta_{123} \Phi^{(1)}_{1,2}},\\
\iota(q_4)=\frac{(\delta_{234}+\delta_{123} \Phi^{(1)}_{1,2}) (\Phi^{(1)}_{1,4}\Phi^{(4)}_{2,3}+\delta_{234} (p_1 q_4-p_4 q_1))}{\delta_{123} \Phi^{(1)}_{1,2} \Phi^{(1)}_{1,3}}.
\end{gather*}

It is convenient to rewrite the first two formulas as:
\[
\frac{1}{\iota(x_4)}=1+\frac{\delta_{123}}{\delta_{234}}\,\Phi^{(1)}_{14},\qquad
\iota(y_4)=\frac{\delta_{124}}{\delta_{234}}\,\Phi^{(1)}_{13}\,\iota(x_4),
\]
and the last one in terms of \(\iota(p_4)\):
\[
\iota(q_4)=\frac{\ \Phi^{(1)}_{14}\,\Phi^{(4)}_{23}+\delta_{234}(p_1q_4-p_4q_1)\ }{\ \Phi^{(1)}_{13}\,\big(\Phi^{(4)}_{14}\,\Phi^{(1)}_{12}-\Phi^{(4)}_{24}\big)\ }\;\iota(p_4).
\]

Let us introduce notation for the four new (rational) invariants:
\[
\zeta^{(4)}_1=\frac{\delta_{123}}{\delta_{234}}\,\Phi^{(1)}_{14},\qquad
\zeta^{(4)}_2=\frac{\delta_{124}}{\delta_{234}}\,\Phi^{(1)}_{13},\qquad
\zeta^{(4)}_3=\frac{\big(\delta_{234}+\delta_{123}\,\Phi^{(1)}_{12}\big)\big(\Phi^{(4)}_{14}\,\Phi^{(1)}_{12}-\Phi^{(4)}_{24}\big)}{\delta_{123}\,\Phi^{(1)}_{12}},
\]
\[
\zeta^{(4)}_4=\frac{\ \Phi^{(1)}_{14}\,\Phi^{(4)}_{23}+\delta_{234}(p_1q_4-p_4q_1)\ }{\ \Phi^{(1)}_{13}\,\big(\Phi^{(4)}_{14}\,\Phi^{(1)}_{12}-\Phi^{(4)}_{24}\big)\ }.
\]
By construction, on the domain of general position, these four quantities are functionally independent and, together with the basis for \(n=3\), form a complete generating set for the field of invariants.

This leads to the following description of the field of invariants:
\begin{theorem}
The field of absolute first-order joint projective invariants is the purely transcendental field
\[
\mathcal{I}_{4,0} \;=\;
\mathbb{R}\bigl(
  \zeta_{12},\zeta_{13},\zeta_{14},\tau,\xi^{(4)}_1,\xi^{(4)}_2,\xi^{(4)}_3,\xi^{(4)}_4\bigr).
\]
The eight listed invariants are algebraically independent and form a transcendence basis for the field \(\mathcal{I}_{4,0}\).
\end{theorem}

The asymmetry in the indices \(1,2,3\) is due to the choice of cross-section (normalization of three points and one pair of derivatives).

\subsection{The Case \boldmath{$n \geq 4$}}

To construct a transcendence basis for the field of absolute joint projective invariants of \(n\) points in general position, we proceed inductively, following the pattern observed in the cases \(n \leq 4\). We introduce the following families of \(4n - 8\) invariants:

\begin{align*}
&Z_2 = \{\zeta_{12}\},\\
&Z_3 := \{\zeta_{12},\; \zeta_{23},\tau\},\\
&Z_k := \{\xi^{(k)}_1,\xi^{(k)}_2,\xi^{(k)}_3,\xi^{(k)}_4\}, \qquad k = 4, \dots, n,
\end{align*}
where
\begin{gather*}
\xi^{(k)}_1=\frac{\delta_{123} \Phi^{(1)}_{1,k}}{\delta_{23k}}, \quad
\xi^{(k)}_2=\frac{\delta_{12k} \Phi^{(1)}_{1,3}}{\delta_{23k}},\\
\xi^{(k)}_3=\frac{(\delta_{23k}+\delta_{123} \Phi^{(1)}_{1,2}) (\Phi^{(k)}_{1,k} \Phi^{(1)}_{1,2}-\Phi^{(k)}_{2,k})}{\delta_{123} \Phi^{(1)}_{1,2}},\\
\xi^{(k)}_4=\frac{(\Phi^{(1)}_{1,k}\Phi^{(k)}_{2,3}+\delta_{23k} (p_1 q_k-p_k q_1))}{ \Phi^{(1)}_{1,3} (\Phi^{(k)}_{1,k} \Phi^{(1)}_{1,2}-\Phi^{(k)}_{2,k})},
\end{gather*}

Denote
\[
\mathcal{G}_n := \bigcup_{k=2}^{n} Z_k.
\]
Then we have the following theorem.

\begin{theorem}
Let \(n \geq 4\). Then the set \(\mathcal{G}_n\) of \(4n-8\) invariants is algebraically independent and forms a transcendence basis for the field \(\mathcal{I}_{n,0}\):
\[
\mathcal{I}_{n,0} = \mathbb{R}\bigl( \mathcal{G}_n \bigr).
\]
\end{theorem}

In \cite{B-2025}, the same fields of invariants are described using a \emph{geometric} approach: the cases \(n=4,5,6\) and \(n \geq 7\) are considered separately, and explicit systems of polynomial generators are constructed. Our approach unifies these cases using the method of moving frames: the same cross-section \(\mathcal{K}\) and invariantization of the blocks \((x_k, y_k, p_k, q_k)\) provide a \emph{unified} formula for all \(n \geq 4\). Moreover, the invariants \(Z_3\) coincide with the geometric ones (up to a rational change of variables), and the new blocks \(Z_k\) (\(k \geq 4\)) are rationally equivalent to the invariants appearing in \cite{B-2025} upon addition of each new point.

\section{Field of Relative Invariants}

The structure of the field $\mathcal{I}_n$ of relative first-order joint projective invariants under the diagonal action of the projective group $PGL(3, \mathbb{R})$ was completely described by direct geometric methods in \cite{B-2025}. It was established that $\mathcal{I}_n$ is a simple algebraic extension of $\mathcal{I}_{n,0}$, and an explicit form of a primitive element of this extension was found for all $n$. 

In this section, we reproduce this result purely by the method of moving frames, in particular obtaining the primitive element of the extension as the invariantization of the Jacobian via the moving frame.

\subsection{What is already known about the field of relative invariants}

A rational function $F \in \mathcal O(\mathcal M)=\mathbb{R}(x_i,y_i,p_i,q_i \mid i=1,\ldots, n)$ is called a \emph{relative invariant of weight} $\omega\in\mathbb{Q}$ with a \emph{multiplier} $\mu:G\times\mathcal M\to\mathbb{R}^\times$ if 
\[
F(g \boldsymbol{x})=\mu(g,\boldsymbol{x})^\omega F(\boldsymbol{x})
\]
for all $g\in G, \boldsymbol{x} \in \mathcal M$.

In this section, we consider the case where $\mu(g,\boldsymbol{x}) = J(g, \boldsymbol{x})$, where $J(g, \boldsymbol{x})$ is the Jacobian determinant of the projective transformation defined by $g$ at the point $\boldsymbol{x}=(x_i,y_i,p_i,q_i \mid i=1,\ldots, n)$. 

Explicitly, 
\[
J(g, \boldsymbol{x})=\frac{D^n}{\displaystyle \prod_{i=1}^n s_i^3}, \quad s_i=c_1 x_i+c_2 y_i+1.
\]

In \cite{B-2025}, it was shown that $\mathcal{I}_n$ is a simple extension of the field of absolute invariants:
\[
\mathcal{I}_n = \mathcal{I}_{n,0}(z_n)
\]
for some primitive element $z_n$ of weight $\frac{1}{g}$, where $g = \gcd(n, 3)$.

An explicit form of the primitive element $z_n$, for $n \geq 5$, was also ''guessed'' in that article:
$$
z_n=
\left(\frac{\displaystyle
      \Bigl(\prod_{i=5}^{\,n} \Delta_{12i}\Bigr)^{3}\;
      \bigl(\Delta_{134}\,\Delta_{234}\bigr)^{\,n-3}}
     {\displaystyle
      \bigl(\Delta_{123}\,\Delta_{124}\bigr)^{\,2n-9}}\right)^{\frac{1}{g}}.
$$

Note that $z_n$ is a rational expression for all $n$. Here, we will find a new expression for the primitive element of the field using the moving frame method.

\subsection{Invariantization of the Jacobian}

In the previous section, we constructed a right moving frame
\[
\rho:\ X\longrightarrow G,\qquad
\rho(g\!\cdot\! \boldsymbol{x})=\rho(\boldsymbol{x})\,g^{-1},
\]
and introduced the invariantization operator for functions on the manifold $X$ by the rule
\[
\iota(F)(\boldsymbol{x}) := F\bigl(\rho(\boldsymbol{x})\!\cdot\!\boldsymbol{x}\bigr).
\]

When the function depends not only on the point but also on the group element, $F:G\times X\to\mathbb{R}$, we must normalize the \emph{group} argument by substituting the parameters defined by the moving frame. That is, the invariantization takes the form:
\begin{equation}\label{def:inv_2}
\iota(F)(g,\boldsymbol{x}) = F\bigl(\rho(\boldsymbol{x}),\,\boldsymbol{x}\bigr).
\end{equation}

We now define the function $\mathcal{C}(\boldsymbol{x})$ as the result of invariantizing the Jacobian:
\[
\mathcal{C}(\boldsymbol{x}) = J\!\bigl(\rho(\boldsymbol{x}),\,\boldsymbol{x}\bigr).
\]

We show that the invariantized Jacobian is a relative invariant of weight $-1$.

\begin{theorem}\label{t4}
For all $g\in G$ and $\boldsymbol{x}\in X$, we have
\[
\mathcal{C}(g\!\cdot\!\boldsymbol{x}) = J(g,\boldsymbol{x})^{-1}\,\mathcal{C}(\boldsymbol{x}).
\]
\end{theorem}

\begin{proof}
We use the multiplicative property of the Jacobian (chain rule): for all $g_1, g_2 \in G$ and $\boldsymbol{x} \in X$,
\begin{equation}\label{eq:jac-mult}
J(g_1 g_2, \boldsymbol{x}) = J\bigl(g_1,\,g_2\!\cdot\!\boldsymbol{x}\bigr)\cdot J(g_2,\boldsymbol{x}).
\end{equation}

Taking $g_1 = \rho(\boldsymbol{x})$ and $g_2 = g^{-1}$ at the point $g\cdot\boldsymbol{x}$ yields
\begin{equation}\label{eq:aux}
J\bigl(\rho(\boldsymbol{x})\,g^{-1},\,g\cdot\boldsymbol{x}\bigr)
= J\bigl(\rho(\boldsymbol{x}),\,\boldsymbol{x}\bigr)\cdot J\bigl(g^{-1},\,g\cdot\boldsymbol{x}\bigr).
\end{equation}

By the definition of the right moving frame, we have $\rho(g\cdot\boldsymbol{x}) = \rho(\boldsymbol{x})\,g^{-1}$, hence
\[
\mathcal{C}(g\cdot\boldsymbol{x}) = J\bigl(\rho(g\cdot\boldsymbol{x}),\,g\cdot\boldsymbol{x}\bigr)
= J\bigl(\rho(\boldsymbol{x})\,g^{-1},\,g\cdot\boldsymbol{x}\bigr).
\]

Substituting \eqref{eq:aux}, we obtain
\[
\mathcal{C}(g\cdot\boldsymbol{x}) = J\bigl(\rho(\boldsymbol{x}),\,\boldsymbol{x}\bigr)\cdot J\bigl(g^{-1},\,g\cdot\boldsymbol{x}\bigr)
= \mathcal{C}(\boldsymbol{x})\cdot J\bigl(g^{-1},\,g\cdot\boldsymbol{x}\bigr).
\]

From \eqref{eq:jac-mult}, taking $g_1 = g$, $g_2 = g^{-1}$, and using $J(1,\boldsymbol{x}) = 1$, we get
\[
J\bigl(g^{-1},\,g\cdot\boldsymbol{x}\bigr) = J(g,\boldsymbol{x})^{-1}.
\]

Thus,
\[
\mathcal{C}(g\cdot\boldsymbol{x}) = J(g,\boldsymbol{x})^{-1}\,\mathcal{C}(\boldsymbol{x}),
\]
which completes the proof.
\end{proof}

Now, for each $n$, define
\[
z'_n = \bigl(\mathcal{C}(\boldsymbol{x})\bigr)^{-\frac{1}{g}}.
\]

By construction, $z'_n$ is a relative invariant of weight $-\frac{1}{g}$ and is therefore the desired primitive element of the extension of the field of absolute invariants:
\[
\mathcal{I}_n = \mathcal{I}_{n,0}(z'_n).
\]

To complete the argument rigorously, we still need to verify that $z'_n$ is a rational expression for all $n$. We address this in the next subsection by computing $z'_n$ explicitly.

\subsection{Explicit formula for the primitive element}


We now compute the invariantized Jacobian and write down explicit formulas for the primitive element of the algebraic extension of the field of absolute invariants.

In practice, invariantization (normalization of the \emph{group} argument) amounts to substituting the parameters from the moving frame $g \gets \rho(\boldsymbol{x})$ into the expression for the Jacobian:
\[
\mathcal C(\boldsymbol{x})
:=J\!\big(\rho(\boldsymbol{x}),\boldsymbol{x}\big)
=\frac{D\!\big(\rho(\boldsymbol{x})\big)^{\,n}}
{\displaystyle\prod_{i=1}^{n}\Big(c_1\!\big(\rho(\boldsymbol{x})\big)\,x_i
+c_2\!\big(\rho(\boldsymbol{x})\big)\,y_i+1\Big)^{3}}.
\]
Here $D(g)$, $c_1(g)$, $c_2(g)$ are the parametric functions of $g\in PGL(3)$, and
$c_j\!\big(\rho(\boldsymbol{x})\big)$ means “the value of the parameter $c_j$ for the element $g=\rho(\boldsymbol{x})$”.

Direct computation yields, up to sign,
\[
\iota\!\left(\frac{D^{3}}{\prod_{i=1}^{3} s_i^{3}}\right)
=\delta_{123}^{-3},
\qquad
s_i=c_1x_i+c_2y_i+1,
\]
and for $i>3$ we have
\[
\iota\!\left(\frac{D}{s_i^{3}}\right)
=\frac{\Phi^{(1)}_{1,2}\,\Phi^{(1)}_{1,3}\,\delta_{123}^{2}}
{\left(\delta_{23i}+\delta_{123}\,\Phi^{(1)}_{1,i}\right)^3}.
\]

Hence, for any $n\ge 3$,
\[
\mathcal C(\boldsymbol{x})
=\iota\!\big(J(g,\boldsymbol{x})\big)
=\frac{1}{\delta_{123}^{3}}
\prod_{i=4}^{n}
\frac{\Phi^{(1)}_{1,2}\,\Phi^{(1)}_{1,3}\,\delta_{123}^{2}}
{\left(\delta_{23i}+\delta_{123}\,\Phi^{(1)}_{1,i}\right)^3}
=\delta_{123}^{\,2n-9}\,
\big(\Phi^{(1)}_{1,2}\,\Phi^{(1)}_{1,3}\big)^{n-3}
\!\!\prod_{i=4}^{n}\frac{1}{\left(\delta_{23i}+\delta_{123}\,\Phi^{(1)}_{1,i}\right)^3}.
\]

Thus we obtain the expression for the primitive element:
\[
z'_n
=\left(\delta_{123}^{\,2n-9}\,
\big(\Phi^{(1)}_{1,2}\,\Phi^{(1)}_{1,3}\big)^{n-3}
\!\!\prod_{i=4}^{n}\frac{1}{\left(\delta_{23i}+\delta_{123}\,\Phi^{(1)}_{1,i}\right)^3} \right)^{-\frac{1}{g}}.
\]

Since $g$ takes only two values, $1$ or $3$, it is easy to see that when $n \equiv 0 \pmod{3}$ the expression in parentheses is a perfect cube of a rational expression, so $z'_n$ is a rational expression, as required.

The identity~\eqref{eq:jac-mult} holds for an arbitrary multiplier $\mu$ (see, in particular,~\cite{Olver1997}), not only for the Jacobian. Accordingly, Theorem~\ref{t4} remains valid in this more general setting: invariantization of the multiplier produces a relative invariant of weight $-1$ of the field $\I^{(\mu)}_n$.


\section{Invariantization Contracting Homotopy}

For a free action of a topological group, the standard bar–cochain complex of continuous cochains is contractible, and therefore all higher cohomologies are trivial.

The goal of this section is to give a \emph{constructive} proof of the triviality of  
$H^m(G,M^\times)$, where $M^{\times}=\Map(X,\Bbb R^{\times})$,  
$G=PGL(3,\mathbb{R})$, and $m\ge 1$,  
by constructing an explicit contracting homotopy on the cochain complex  
$C^\bullet(G,M^\times)$, using the global right moving frame built in the previous sections
\[
\rho:\ \mathcal M\longrightarrow G,\qquad \rho(g\!\cdot\!\boldsymbol x)=\rho(\boldsymbol x)\,g^{-1}.
\]
In particular, using the moving frame, we define operators
\[
h^{\,m}:C^m\longrightarrow C^{m-1}\qquad(m\ge1),
\]
for which the multiplicative identity
\[
d^{m-1}h^{\,m}\ \cdot\ h^{\,m+1}d^m\;=\;\mathrm{id}_{C^m}
\]
holds.  
It follows immediately that for every $m$–cocycle $c\in Z^m(G,M^\times)$ we have
\[
c\;=\;d^{m-1}\!\bigl(h^{\,m}c\bigr),
\]
that is, \emph{invariantization via the moving frame} provides an \emph{explicit} $(m\!-\!1)$–coboundary generating the given $m$–cocycle, and hence $Z^m=B^m$ for all $m\ge 1$.

Below we fix the multiplicative definitions of the cochain complex, the differentials, and normalization, following the additive conventions of \cite{Braun}, and then explicitly implement the described invariantization–homotopy construction.

\subsection{Multiplicative cochain complex for the action of the projective group}

Let
\[
M^\times:=\Map(\mathcal M,\Bbb R^\times),
\]
the multiplicative abelian group (under pointwise multiplication) of all nonzero rational functions on $\mathcal M$.  
We consider the standard inhomogeneous \emph{multiplicative} cochain complex with coefficients in $M^\times$:
\[
(C^\bullet,d^\bullet):\qquad
C^0 \xrightarrow{\,d^0\,} C^1 \xrightarrow{\,d^1\,} C^2 \xrightarrow{\,d^2\,} \cdots,
\qquad
C^m:=\Map\!\big(G^m\times\mathcal M,\Bbb R^\times\big).
\]
For the action of $G$ on $M$ the boundary operators are given by
\[
(d^0 c)(g,\boldsymbol x)=\frac{c(g\!\cdot\!\boldsymbol x)}{c(\boldsymbol x)},
\qquad c\in C^0=M^\times,
\]
\[
(d^1 c)(g_1,g_2,\boldsymbol x)
=\frac{c(g_1,g_2\!\cdot\!\boldsymbol x)\,c(g_2,\boldsymbol x)}{c(g_1 g_2,\boldsymbol x)},
\qquad c\in C^1,
\]

and in general, for $m\ge 1$ and $c\in C^m$,
\begin{multline}
(d^m c)(g_1, \ldots, g_{m+1}; \boldsymbol x) =\\=
c(g_2, \ldots, g_{m+1}; \boldsymbol x) 
\cdot \prod_{i=1}^{m} 
c(g_1, \ldots, g_i g_{i+1}, \ldots, g_{m+1}; \boldsymbol x)^{(-1)^i} 
\cdot c(g_1, \ldots, g_m; g_{m+1} \cdot \boldsymbol x)^{(-1)^{m+1}}.
\end{multline}

As usual:
\[
Z^m(G,M^\times)=\ker d^m,\qquad
B^m(G,M^\times)=\operatorname{Im} d^{m-1},\qquad
H^m(G,M^\times)=Z^m/B^m.
\]

We have
\[
Z^0(G,M^\times)=\{\,c\in M^\times:\ c(g\!\cdot\!\boldsymbol x)=c(\boldsymbol x)\ \ \forall g,\boldsymbol x\,\}
=(M^\times)^G,
\]
that is, exactly the field of \emph{absolute invariants} of the action of $G$ on $\mathcal M$, explicitly described in Section~2. Hence
\[
H^0(G,M^\times)=(M^\times)^G.
\]

The group $B^1(G,M^\times)$ consists of \emph{$1$–coboundaries}, that is, functions of the form
\[
\mu(g,\boldsymbol x)=(d^0 c)(g,\boldsymbol x)=\frac{c(g\!\cdot\!\boldsymbol x)}{c(\boldsymbol x)},
\qquad c\in M^\times.
\]
Such $\mu$ are interpreted as \emph{multipliers} of a relative invariant $c$ of the group $G$, i.e.
$c(g\!\cdot\!\boldsymbol x)=\mu(g,\boldsymbol x)\,c(\boldsymbol x)$.

The group $Z^1(G,M^\times)$ consists of \emph{$1$–cocycles} $b:G\times\mathcal M\to\Bbb K^\times$ satisfying
\begin{equation}\label{eq:cocycle}
c(g_1 g_2,\boldsymbol x)=c(g_1,\,g_2\!\cdot\!\boldsymbol x)\,c(g_2,\boldsymbol x),\qquad 
c(1,\boldsymbol x)=1,\quad \forall g_1,g_2\in G,\ \boldsymbol x\in\mathcal M.
\end{equation}
In particular, the Jacobian of the prolonged action $J(g,\boldsymbol x)$ is a $1$–cocycle.

A \emph{contracting homotopy} of the cochain complex $(C^\bullet,d^\bullet)$ is a family of maps
$h^\bullet=\{h^m:C^m\to C^{m-1}\}_{m\ge1}$ satisfying, in multiplicative notation,
\begin{equation}\label{def:homotopy}
\bigl(d^{m-1}\circ h^m\bigr)\ \cdot\ \bigl(h^{m+1}\circ d^m\bigr)\;=\;\mathrm{id}_{C^m}\qquad(m\ge1).
\end{equation}
From the existence of a contracting homotopy it follows immediately that the complex is \emph{acyclic} in all degrees $m\ge1$, that is,
\[
H^m(G,M^\times)=1.
\]

\subsection{Triviality of $H^1(G,M^\times)$}

Fix a global right moving frame
\[
\rho:\ \mathcal M\longrightarrow G,\qquad \rho(g\!\cdot\!\boldsymbol x)=\rho(\boldsymbol x)\,g^{-1},
\]
and define the invariantization operator for $1$–cochains
\begin{equation}\label{def:iota}
\iota:\ C^1 \longrightarrow C^0,
\qquad
\iota(\lambda)(\boldsymbol x):=\lambda\bigl(\rho(\boldsymbol x),\boldsymbol x\bigr).
\end{equation}

\begin{theorem}\label{thm:Z1=B1}
For every $1$–cocycle $b\in Z^1(G,M^\times)$ the function
\[
B_\rho(\boldsymbol x):=\iota(b)(\boldsymbol x)^{-1}=\frac{1}{b\bigl(\rho(\boldsymbol x),\boldsymbol x\bigr)}\in M^\times
\]
satisfies
\[
(d^0 B_\rho)(g,\boldsymbol x)=\frac{B_\rho(g\!\cdot\!\boldsymbol x)}{B_\rho(\boldsymbol x)}=b(g,\boldsymbol x)\qquad \forall\,g\in G,\ \boldsymbol x\in\mathcal M.
\]
Hence $b=d^0 B_\rho$, that is, $Z^1(G,M^\times)=B^1(G,M^\times)$ and, in particular, $H^1(G,M^\times)=1$.
\end{theorem}

\begin{proof}
By definition,
\[
(d^0 B_\rho)(g,\boldsymbol x)
=\frac{b(\rho(g\!\cdot\!\boldsymbol x),\,g\!\cdot\!\boldsymbol x)^{-1}}{b(\rho(\boldsymbol x),\,\boldsymbol x)^{-1}}
=\frac{b(\rho(\boldsymbol x),\,\boldsymbol x)}{b(\rho(\boldsymbol x)g^{-1},\,g\!\cdot\!\boldsymbol x)}.
\]
The cocycle condition \eqref{eq:cocycle} for the pair $(\rho(\boldsymbol x),g^{-1})$ gives
$$
b(\rho(\boldsymbol x)g^{-1},\,g\!\cdot\!\boldsymbol x)
=b(\rho(\boldsymbol x),\,\boldsymbol x)\,b(g^{-1},\,g\!\cdot\!\boldsymbol x).
$$
Therefore
\[
(d^0 B_\rho)(g,\boldsymbol x)=b(g^{-1},\,g\!\cdot\!\boldsymbol x)^{-1}=b(g,\boldsymbol x),
\]
where the last equality follows from \eqref{eq:cocycle} with $h=g^{-1}$.
\end{proof}

\medskip

For an arbitrary $\lambda\in C^1$, not necessarily a cocycle, from \eqref{def:iota} we obtain an explicit “defect” formula:
\begin{equation*}
(d^0\,\iota(\lambda)^{-1})(g,\boldsymbol x)
=\frac{\lambda(\rho(\boldsymbol x),\boldsymbol x)}{\lambda(\rho(\boldsymbol x)g^{-1},\,g\!\cdot\!\boldsymbol x)}
=\lambda(g,\boldsymbol x)\cdot\bigl(d^1\lambda\bigr)\!\bigl(\rho(g \boldsymbol x),g;\boldsymbol x\bigr)^{-1},
\end{equation*}
or
\begin{equation}\label{eq:defect}
(d^0\,\iota(\lambda)^{-1})(g,\boldsymbol x) \cdot\bigl(d^1\lambda\bigr)\!\bigl(\rho(g \boldsymbol x),g;\boldsymbol x\bigr)
=\lambda(g,\boldsymbol x).
\end{equation}

Thus $d^0\,\iota(\lambda)^{-1}=\lambda$ if and only if $\lambda\in Z^1(G,M^\times)$.

For any $a\in C^0$ (that is, $\lambda=d^0 a\in B^1$) invariantization is compatible with $d^0$ in the sense that
\[
\iota(d^0 a)(\boldsymbol x)=\frac{a\bigl(\rho(\boldsymbol x)\!\cdot\!\boldsymbol x\bigr)}{a(\boldsymbol x)},
\]
so the composition $\iota\circ d^0$ recovers $a$ up to a factor from $(M^\times)^G=\ker d^0$.

\medskip

Let
\[
s:Z^1 \longrightarrow C^0,\qquad s(b):=\iota(b)^{-1}=b\bigl(\rho(\boldsymbol x),\boldsymbol x\bigr)^{-1}.
\]
Then by Theorem~\ref{thm:Z1=B1} we have $d^0\circ s=\mathrm{id}_{Z^1}$.  
This is compactly expressed by the commutative diagram
\[
\xymatrix@C=3em@R=2em{
& Z^1 \ar[dl]_{\,s} \ar[dr]^{\mathrm{id}} & \\
C^0 \ar[rr]_{d^0} && Z^1
}
\]
that is,
\[
d^0\circ s=\mathrm{id}_{Z^1}.
\]

Let \(\pi:C^0\to C^0/(M^\times)^G\) be the natural projection modulo \(\ker d^0\).
Then for $B^1=\operatorname{Im} d^0$ we have the commutative square
\[
\xymatrix@C=4em@R=3em{
C^0 \ar[r]^{\ d^0\ } \ar[d]_{\ \pi} &
B^1 \ar[d]^{\ \pi\circ s} \\
C^0/(M^\times)^G \ar[r]_{\ \ \mathrm{id}\ } &
C^0/(M^\times)^G
}
\]
that is,
\[
(\pi\circ s\circ d^0)(a)=\pi(a)\qquad \forall\,a\in C^0.
\]

\medskip

For better insight we now consider the case $m=2$, in which the patterns of the general acyclicity proof already become visible.

\subsection{Triviality of $H^2(G,M^\times)$.}

Define the homotopy operators $h^2, h^3$, which are based on “freezing” the first argument of a cochain using the value of the moving frame. For $c\in C^{2}$ set
\begin{equation}\label{eq:def-h2}
h^{2}\colon C^{2}\to C^{1},\qquad
\big(h^{2}c\big)(g;\boldsymbol x):=c\bigl(\rho(g\cdot\boldsymbol x),\,g;\,\boldsymbol x\bigr).
\end{equation}
Similarly, for $\Lambda\in C^{3}$ define
\begin{equation}\label{eq:def-h3}
h^{3}\colon C^{3}\to C^{2},\qquad
\big(h^{3}\Lambda\big)(g_1,g_2;\boldsymbol x)
:=\Lambda\bigl(\rho(g_1 g_2 \cdot\boldsymbol x),\,g_1,\,g_2;\,\boldsymbol x\bigr).
\end{equation}

We now show that \eqref{eq:def-h2}–\eqref{eq:def-h3} satisfy the condition
\[
d^{1}h^{2}\ \cdot\ h^{3}d^{2}\;=\;\mathrm{id}_{C^{2}},
\]
which is the fragment of the contracting homotopy $h^\bullet$ on the entire cochain complex
$C^\bullet(G,M^\times)$.

\begin{theorem}\label{lem:defect-m2}
For an arbitrary cochain $c\in C^{2}$ the following multiplicative identity holds:
\begin{equation}\label{eq:defect-m2}
(d^{1}h^{2}c)(g_1,g_2;\boldsymbol x)
\cdot \bigl(h^{3}d^{2}c\bigr)(g_1,g_2;\boldsymbol x)=\;
c(g_1,g_2;\boldsymbol x),
\qquad
\forall\,g_1,g_2\in G,\ \boldsymbol x\in\mathcal M.
\end{equation}
\end{theorem}

\begin{proof}
To simplify notation introduce an auxiliary symbol for the value of the moving frame at the shifted point:
\[
R=\rho(g_1 g_2\cdot\boldsymbol x).
\]
From the equivariance property of the moving frame we have $R=\rho(\boldsymbol x)(g_1g_2)^{-1}$, hence $Rg_1 = \rho(\boldsymbol x)g_2^{-1} = \rho(g_2\cdot\boldsymbol x)$.

Then, by the definition of the coboundary $d^1$ and the operator $h^2$, we have
\begin{gather*}
(d^{1}h^{2}c)(g_1,g_2;\boldsymbol x)
=\frac{(h^{2}c)(g_1;\,g_2\!\cdot\!\boldsymbol x)\cdot (h^{2}c)(g_2;\boldsymbol x)}{(h^{2}c)(g_1g_2;\boldsymbol x)}
=\\=\frac{c\bigl(\rho(g_1 g_2\cdot\boldsymbol x),\,g_1;\,g_2\!\cdot\!\boldsymbol x\bigr)\cdot c\bigl(\rho(g_2\cdot\boldsymbol x),\,g_2;\,\boldsymbol x\bigr)}{c\bigl(\rho(g_1g_2\cdot\boldsymbol x),\,g_1g_2;\,\boldsymbol x\bigr)}=
\frac{c(R,\,g_1;\,g_2\!\cdot\!\boldsymbol x)\cdot c(Rg_1,\,g_2;\,\boldsymbol x)}{c(R,\,g_1g_2;\,\boldsymbol x)}.
\end{gather*}

Now compute the result of applying the homotopy to the coboundary $(h^{3}d^{2}c)$.  
By definition of the operator $h^3$ (see \eqref{eq:def-h3}), this is the value of the $2$–coboundary $(d^2 c)$ at the triple of arguments $(R,\, g_1,\, g_2)$:
$$
 (h^{3}d^{2}c)(g_1,g_2;\boldsymbol x) = (d^{2}c)(R,\,g_1,\,g_2;\,\boldsymbol x)
 =\frac{c(g_1,g_2;\,\boldsymbol x)\cdot c(R,\,g_1g_2;\,\boldsymbol x)}{c(Rg_1,\,g_2;\,\boldsymbol x)\cdot c(R,\,g_1;\,g_2\!\cdot\!\boldsymbol x)}.
$$

Multiplying the obtained expressions for $(d^{1}h^{2}c)(g_1,g_2;\boldsymbol x)$ and $(h^{3}d^{2}c)(g_1,g_2;\boldsymbol x)$ yields the required result.
\end{proof}

As a consequence we immediately obtain the triviality of $H^2(G,M^\times)$.

\subsection{Contracting Homotopy for Arbitrary $m$.}

For each $m\ge 1$ define the homotopy operator
$h^m\colon C^m\to C^{m-1}$ by
\begin{equation}\label{eq:hm-def-final}
\bigl(h^{m}c\bigr)(g_1,\ldots,g_{m-1};\boldsymbol x)
\;:=\;
c\bigl(\rho(g_1 \cdots g_{m-1} \!\cdot\! \boldsymbol x),\,g_1,\ldots,g_{m-1};\,\boldsymbol x\bigr),
\qquad c\in C^m.
\end{equation}
For $m=2,3$ this agrees with the previously introduced definitions of $h^2,h^3$.

\begin{theorem}
For every $m\ge 1$ and any cochain $c\in C^{m}$ the following multiplicative identity holds:
\begin{equation*}
\bigl(d^{m-1}h^{m}c\bigr)(g_1,\ldots,g_m;\boldsymbol x)\,\cdot\,
\bigl(h^{m+1}d^{m}c\bigr)(g_1,\ldots,g_m;\boldsymbol x)
\;=\;
c(g_1,\ldots,g_m;\boldsymbol x).
\end{equation*}
\end{theorem}

\begin{proof}
Fix $m\ge 1$, a point $\boldsymbol x\in\mathcal M$, and set
\[
R = \rho(g_1\cdots g_m\!\cdot\!\boldsymbol x)\in G.
\]
From the equivariance property of the moving frame we have
$R=\rho(\boldsymbol x)(g_1g_2\cdots g_m)^{-1}$, and therefore obtain the key relation
\begin{equation}\label{Rg}
Rg_1 = \rho(\boldsymbol x)(g_2 \cdots g_m)^{-1} 
      = \rho(g_2 \cdots g_m\cdot\boldsymbol x).
\end{equation}

We first compute the factor ${h^{m+1}d^{m}c}$.  
By definition~\eqref{eq:hm-def-final},
\begin{gather*}
\bigl(h^{m+1}d^m c\bigr)(g_1,\ldots,g_m;\boldsymbol x)
=
\bigl(d^m c\bigr)\bigl(\rho(g_1\cdots g_m\!\cdot\!\boldsymbol x),\,g_1,\ldots,g_m;\boldsymbol x\bigr)
=
(d^m c)(R,g_1,\ldots,g_m;\boldsymbol x)=\\[4pt]
=
c(g_1, \ldots, g_{m}; \boldsymbol x)
\cdot 
c(R g_1, \ldots, g_i g_{i+1}, \ldots, g_{m}; \boldsymbol x)^{-1} 
\times \\[4pt]
\times 
\prod_{i=2}^{m} 
c(R,g_1, \ldots, g_i g_{i+1}, \ldots, g_{m}; \boldsymbol x)^{(-1)^i} 
\cdot 
c(g_1, \ldots, g_{m-1}; g_{m} \cdot \boldsymbol x)^{(-1)^{m+1}}.
\end{gather*}

\medskip

Now compute the factor ${d^{m-1}h^{m}c}$.
\begin{gather*}
\bigl(d^{m-1} h^{m}c\bigr)(g_1,\ldots,g_m;\boldsymbol x)
=
(h^{m}c)(g_2, \ldots, g_{m}; \boldsymbol x) 
\cdot 
\prod_{i=1}^{m-1} 
(h^{m}c)(g_1, \ldots, g_i g_{i+1}, \ldots, g_{m}; \boldsymbol x)^{(-1)^i} 
\\[4pt]
\times\ 
(h^{m}c)(g_1, \ldots, g_{m-1}; g_{m} \cdot \boldsymbol x)^{(-1)^{m}}.
\end{gather*}

We now expand each of the factors using \eqref{Rg}.

For the first factor:
\begin{gather*}
(h^{m}c)(g_2, \ldots, g_{m}; \boldsymbol x)
=
c(\rho(g_2 \cdots g_{m} \cdot \boldsymbol x),g_2, \ldots, g_{m}; \boldsymbol x)
=
c(Rg_1,g_2, \ldots, g_{m}; \boldsymbol x),
\end{gather*}
and this same factor occurs in the expansion of ${h^{m+1}d^{m}c}$.

Similarly, for $1\le i \le m-1$, noting that the product $g_1\cdots g_i g_{i+1}\cdots g_m$
coincides with $g_1\cdots g_m$, we get
\begin{gather*}
(h^{m}c)(g_1, \ldots, g_i g_{i+1}, \ldots, g_{m}; \boldsymbol x)^{(-1)^i}
=
c(\rho(g_1\cdots g_m\!\cdot\!\boldsymbol x),g_1, \ldots, g_i g_{i+1}, \ldots, g_{m}; \boldsymbol x)^{(-1)^i}
\\[4pt]
=
c(R, g_i g_{i+1}, \ldots, g_{m}; \boldsymbol x)^{(-1)^i}.
\end{gather*}

Finally,
\begin{gather*}
(h^{m}c)(g_1, \ldots, g_{m-1}; g_{m} \cdot \boldsymbol x)^{(-1)^{m}}
=
c(\rho(g_1\cdots g_{m}\cdot \boldsymbol x),g_1, \ldots, g_{m-1}; g_{m} \cdot \boldsymbol x)^{(-1)^{m}}
\\[4pt]
=
c(R,g_1, \ldots, g_{m-1}; g_{m} \cdot \boldsymbol x)^{(-1)^{m}}.
\end{gather*}

Thus, in the product
\[
\bigl(d^{m-1}h^{m}c\bigr)(g_1,\ldots,g_m;\boldsymbol x)\,\cdot\,
\bigl(h^{m+1}d^{m}c\bigr)(g_1,\ldots,g_m;\boldsymbol x)
\]
all terms cancel pairwise, leaving only the single factor
\[
c(g_1,\ldots,g_m;\boldsymbol x),
\]
which proves the required identity.

\end{proof}

As a consequence we obtain
\[
H^m(G,M^\times)=1,\qquad m\ge1.
\]


\begin{thebibliography}{30}

\bibitem{Halphen1878} 
Halphen, M. (1878). \textit{Sur les invariants differentiels} [Doctoral dissertation]. Paris.

\bibitem{Lovett1898} 
Lovett, E. O. (1898). Certain invariants of a quadrangle by projective transformation. \textit{Annals of Mathematics, 12}(1/6), 79--86.

\bibitem{Lovett1899} 
Lovett, E. O. (1899). Note on differential invariants of a system of $m$ points by projective transformation. \textit{American Journal of Mathematics, 21}(2), 168--174.

\bibitem{Wilczynski1906}
Wilczynski, E. J. (1906). \textit{Projective differential geometry of curves and ruled surfaces}. Leipzig
- Teubner.

\bibitem{Bouton1898}
Bouton, C. L. (1898). Some examples of differential invariants. \textit{Bulletin of the American Mathematical Society, 4}(7), 313--322.

\bibitem{FSB}
Flusser, J., Suk, T., \& Zitova, B. (2017). \textit{2D and 3D image analysis by moments}. John Wiley \& Sons.

\bibitem{Olver2023}
Olver, P. J. (2023). Projective invariants of images. \textit{European Journal of Applied Mathematics, 34}(5), 936--946.

\bibitem{Wang}
Wang, Y. B., Wang, X. W., Zhang, B., \& Wang, Y. (2015). Projective invariants of D-moments of 2D grayscale images. \textit{Journal of Mathematical Imaging and Vision, 51}(2), 248--259.

\bibitem{5}
Wang, Y. B., Wang, X. W., Zhang, B., \& Wang, Y. (2015). Projective invariants of D-moments of 2D grayscale images. \textit{Journal of Mathematical Imaging and Vision, 51}(2), 248--259.

\bibitem{1} 
Li, E., Mo, H., Xu, D., \& Li, H. (2019). Image projective invariants. \textit{IEEE Transactions on Pattern Analysis and Machine Intelligence, 41}, 1144--1157.

\bibitem{Open}
Li, E., Mo, H., Xu, D., \& Li, H. (2019). Image projective invariants. \textit{IEEE Transactions on Pattern Analysis and Machine Intelligence, 41}, 1144--1157.

\bibitem{B-2025} 
Bedratyuk, L. (2025). First order joint differential projective invariants. \textit{Communications in Algebra}, 1--20. https://doi.org/10.1080/00927872.2025.2569449

\bibitem{Olver1999-1}
Fels, M., \& Olver, P. J. (1999). Moving coframes. II. Regularization and theoretical foundations. \textit{Acta Applicandae Mathematicae, 55}, 127--208.

\bibitem{Olver1999-2} 
Olver, P. J. (1999). \textit{Classical invariant theory} (Vol. 44). Cambridge University Press.

\bibitem{Olver2007} 
Olver, P. J. (2007). Generating differential invariants. \textit{Journal of Mathematical Analysis and Applications, 333}, 450--471.

\bibitem{Olver2011}
Olver, P. J. (2011). Lectures on moving frames. In D. Levi, P. Olver, Z. Thomova, \& P. Winternitz (Eds.), \textit{Symmetries and integrability of difference equations} (pp. 207--246). Cambridge University Press.

\bibitem{Man}
Mansfield, E. L. (2010). \textit{A practical guide to the invariant calculus} (Vol. 26). Cambridge University Press.

\bibitem{Rot}
Rothe, I., Susse, H., \& Voss, K. (1996). The method of normalization to determine invariants. \textit{IEEE Transactions on Pattern Analysis and Machine Intelligence, 18}(4), 366--376.

\bibitem{Ros}
Rosenlicht, M. (1963). A remark on quotient spaces. \textit{Anais da Academia Brasileira de Ci?ncias, 35}, 487--489.

\bibitem{Sp}
Springer, T. A. (1977). \textit{Invariant theory} (Vol. 585). Springer.

\bibitem{Olver1997}
Fels, M., \& Olver, P. J. (1997). On relative invariants. \textit{Mathematische Annalen, 308}(4), 701--730.

\bibitem{Braun}
Brown, K. S. (1982). \textit{Cohomology of groups} (Vol. 87). Springer-Verlag.

\end{thebibliography}
\end{document}